\documentclass{article}
\usepackage{amssymb,amsmath,amsthm,graphicx,kotex} 

\usepackage{color}

\def\qed{\hfill {\hbox{${\vcenter{\vbox{               
   \hrule height 0.4pt\hbox{\vrule width 0.4pt height 6pt
   \kern5pt\vrule width 0.4pt}\hrule height 0.4pt}}}$}}}

\def\utr{\underline\triangleright}
\def\otr{\overline\triangleright}

\newtheorem{theorem}{Theorem}

\theoremstyle{definition}
\newtheorem{example}{Example}
\newtheorem{definition}{Definition}
\newtheorem{remark}{Remark}

\date{}

\title{\Large \textbf{MC-Biquandles and MC-Biquandle Coloring Quivers}}

\author{
Seonmi Choi\footnote{Email: smchoi@knu.ac.kr. Partially supported by Basic Science Research Program through the National Research Foundation of Korea (NRF) funded by the Ministry of Education (No. 2021R1I1A1A01049100) and the National Research Foundation of Korea (NRF) grant funded by the Korean government (MSIT) (No. 2022R1A5A1033624).}
 \and
Sam Nelson\footnote{Email: Sam.Nelson@cmc.edu. Partially supported by
Simons Foundation Collaboration Grant 702597.}}

\begin{document}
\maketitle

\begin{abstract}
We introduce the notion of \textit{mc-biquandles}, algebraic structures 
which have possibly distinct biquandle operations at single-component and 
multi-component crossings. These structures provide computable homset 
invariants for classical and virtual links. We categorify these homsets 
to obtain \textit{mc-biquandle coloring quivers} and define several new link 
invariants via decategorification from these invariant quivers.
\end{abstract}

\parbox{6in} {\textsc{Keywords:} mc-biquandles, coloring quivers, 
link invariants, categorification

\smallskip

}

\section{\large\textbf{Introduction}}\label{I}

\textit{Biquandles} are algebraic structures with the useful property that
the set of biquandle homomorphisms, known as \textit{homsets} can be used to define
invariants of oriented knots and links. Elements of a biquandle homset 
$\mathrm{Hom}(\mathcal{B}(L),X)$ from the \textit{fundamental biquandle} of 
an oriented knot or link $L$ to a finite biquandle $X$ can be 
visualized as \textit{biquandle colorings} of a choice of diagram $D$ for $L$,
analogously to representing a linear transformation as a matrix by making 
a choice of bases for the input and output spaces.

In \cite{NOS}, the biquandle structure was extended to the case of 
\textit{psyquandles} in order to define coloring invariants for pseudoknots and
singular knots. This extension involved 
using the standard biquandle operations at classical crossings 
and introducing new operations at the precrossings or singular crossings.
In this paper, we use a similar approach to strengthen biquandle
counting invariants for multi-component links. More precisely, we define 
different biquandle-style operations at multi-component crossings and at
single-component crossings and call the resulting structure \textit{mc-biquandles}. 
An mc-biquandle can be understood as a pair consisting of a standard biquandle
for use at single-component crossings and a birack for use at multi-component
crossings with some interaction conditions as required by Reidemeister III moves
with mixed crossing types. Our construction does not
depend on the genus of the supporting surface for the link diagram, and
hence applies to the case of virtual links as well.

In \cite{CN}, a directed graph or \textit{quiver} structure was defined on 
the homset $\mathrm{Hom}(\mathcal{Q}(L),X)$ of homomorphisms from the
knot quandle to a finite target quandle $X$, which is invariant under Reidemeister
moves. Since quivers are categories, this construction is a kind of 
categorification of the quandle homset invariant. In later papers, such
as \cite{CCN}, the quiver construction was extended to other 
algebraic homset invariants. In this paper, we expand upon this concept to 
apply it to our novel structure of mc-biquandles, providing categorifications 
of the mc-biquandle homset invariants. Moreover, we define new polynomial 
and multiset-valued invariants of classical and virtual links based on these
quivers as decategorifications.

The paper is organized as follows. In Section \ref{MCB}, we define 
\textit{mc-biquandles} with different operations at single-component and 
multi-component crossings. In Section \ref{MCBQ}, we categorify
this structure to obtain mc-biquandle coloring quivers. 
Our main results are that these new algebraic structures define invariants
of oriented classical and virtual knots and links.
Throughout the paper 
we provide examples to illustrate the 
computations of the invariants and to show some tables of results for examples 
of invariants in this new infinite family. We conclude in Section \ref{Q} with 
some questions for future research.

\section{\large\textbf{Biquandles and MC-Biquandles}}\label{MCB}

We begin by recalling a definition (see \cite{EN} for more).

\begin{definition} 
A \textit{biquandle} is a set $X$ with a pair of binary 
operations $\utr,\otr$  satisfying
\begin{itemize}
\item[(i)] For all $x\in X$, $x\otr x=x\utr x$,
\item[(ii)] For all $x,y\in X$, the maps 
$\alpha_y,\beta_y:X\to X$
and $S:X\times X\to X\times X$ defined by
\[\alpha_y(x)=x\utr y, \
\beta_y(x)=x\otr y, \ \mathrm{and~} 
S(x,y)=(y\otr x,x\utr y)
\] are invertible, and
\item[(iii)] For all $x,y,z\in X$, we have the 
\textit{exchange laws}
\[\begin{array}{rcl}
(x\otr y)\otr (z\otr y) & = & (x\otr z)\otr (y\utr z) \\
(x\utr y)\otr (z\utr y) & = & (x\otr z)\utr (y\otr z) \\
(x\utr y)\utr (z\utr y) & = & (x\utr z)\utr (y\otr z) \\
\end{array}.\]
\end{itemize}
A set $X$ with binary operations satisfying axioms (ii) and (iii) is called a 
\textit{birack}.
\end{definition}

The biquandle axioms are the conditions required by the Reidemeister moves
so that for every \textit{biquandle coloring} of an oriented knot or link
diagram, i.e., an assignment of elements of $X$ to the semiarcs in a diagram 
such that the condition
\[\includegraphics{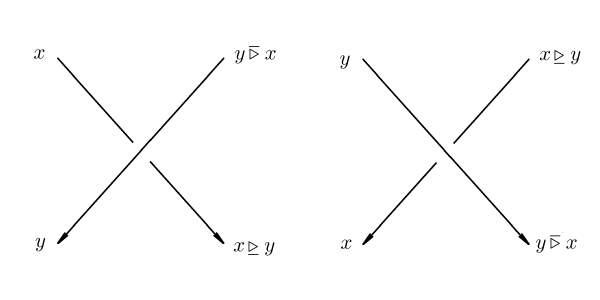}\]
holds at every crossing, there exists a unique biquandle coloring of the 
diagram after the move which agrees with the coloring before the move 
outside the neighborhood of the move.

We next observe that in an oriented link diagram, crossings fall into two
disjoint classes: \textit{single-component}, where both strands in the 
crossing are from the same component of the link, and \textit{multi-component},
where the two strands in the crossing are from different components of the link.
Moreover, crossings cannot switch between these sets during Reidemeister moves.
We can then generalize the biquandle structure to have distinct operations 
at single-component crossings and multi-component crossings.

\begin{definition}\label{Defmcbiqdle} 
An \textit{mc-biquandle} is a set $X$ with two pairs of binary 
operations $\utr^s,\otr^s,\utr^m$ and $\otr^m$ satisfying
\begin{itemize}
\item[(i)] For all $x\in X$, $x\otr^s x=x\utr^s x$,
\item[(ii)] For all $x,y\in X$ and $j\in \{s,m\}$, the maps 
$\alpha^j_y,\beta^j_y:X\to X$
and $S^j:X\times X\to X\times X$ defined by
\[\alpha^j_y(x)=x\utr^j y, \
\beta^j_y(x)=x\otr^j y, \ \mathrm{and~} 
S^j(x,y)=(y\otr^j x,x\utr^j y)
\] are invertible, and
\item[(iii)] For all $x,y,z\in X$ and 
$(j,k,l)\in\{(s,s,s),(s,m,m), (m,s,m),(m,m,s),(m,m,m)\}$, we have the 
\textit{multi-component exchange laws}
\[\begin{array}{rcl}
(x\otr^k y)\otr^l (z\otr^j y) & = & (x\otr^l z)\otr^k (y\utr^j z) \\
(x\utr^l y)\otr^j (z\utr^k y) & = & (x\otr^j z)\utr^l (y\otr^k z) \\
(x\utr^k y)\utr^j (z\utr^l y) & = & (x\utr^j z)\utr^k (y\otr^l z) \\
\end{array}.\]
\end{itemize}
\end{definition}

Let us label single-component crossings with $s$ and multi-component crossings
with $m$. Then an \textit{mc-biquandle coloring} of a classical or virtual knot
or link diagram is an assignment of elements of $X$ to the semiarcs in the 
diagram such that at every crossing we have the following:

\[\includegraphics{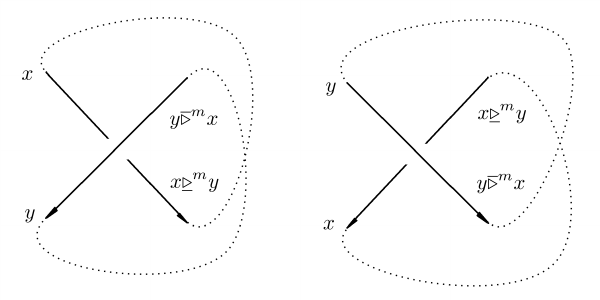}\]
\[\includegraphics{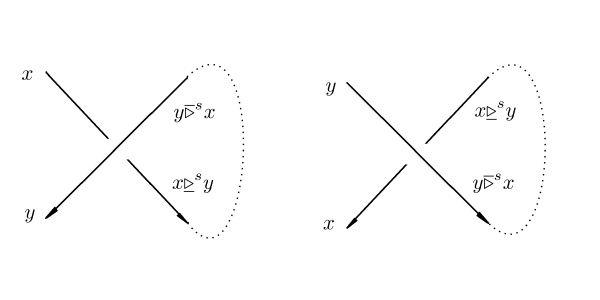}\]

We then observe that Reidemeister I moves involve 
only one single-component crossing, 
\[\includegraphics{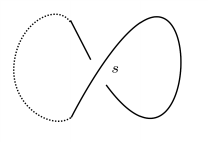}\]
Reidemeister II moves involve two single-component
or two multi-component crossings, 
\[\includegraphics{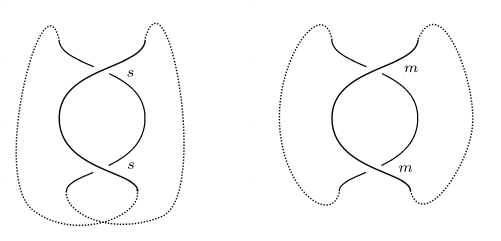}\]
and Reidemeister III moves can involve 
three single-component crossings, two multi-component crossings and one single-component 
crossing, or three multi-component crossings:
\[\includegraphics{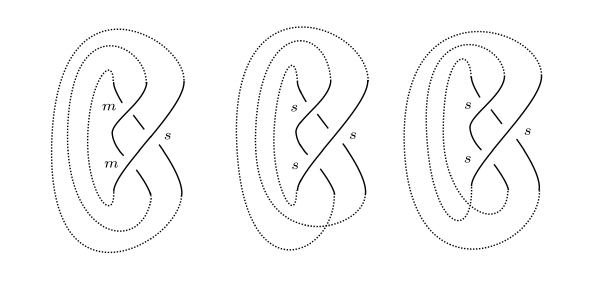}\]
\[\includegraphics{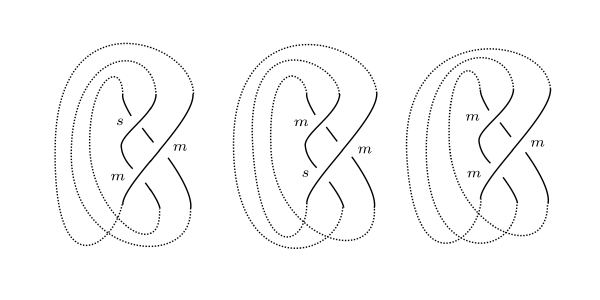}\]
Then using the principle that for each labeling
of semiarcs on one side of a Reidemeister move, there should be a unique 
labeling of semiarcs on the other side of the move which agrees with the
labeling before the move on the boundary of the neighborhood of the move,
we obtain the mc-biquandle axioms. For example, the Reidemeister III move
yields the multi-component exchange laws:
\[\includegraphics{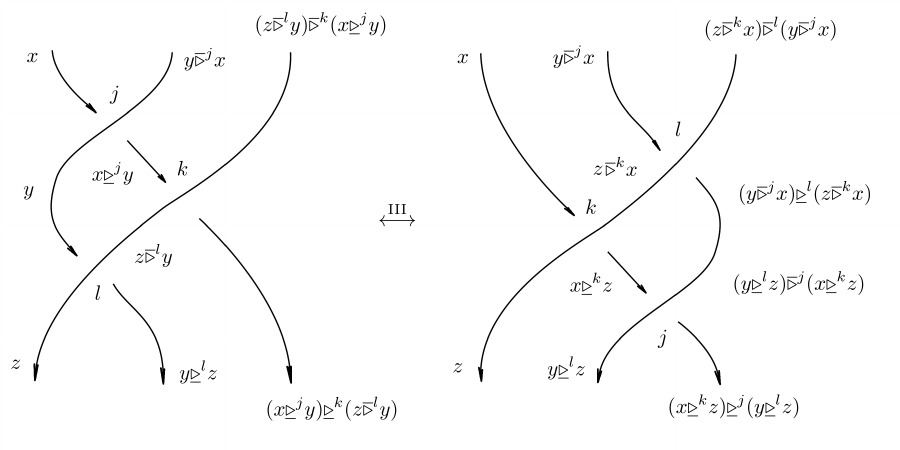}\]

\begin{remark}
Note that these operations are defined not on the geometric semiarcs themselves, 
but on the algebraic coloring structure used to color the diagram. While it is 
true that at any given crossing the strands involved are both single-component 
or both multi-component, the \textit{colors} used to label them are neither
fundamentally single-component nor multi-component, so it makes sense to 
define both kinds of operations on the same set. This setup is similar to 
many other structures in the literature including \textit{parity biquandles}
in \cite{KK, KN}, \textit{virtual (bi)quandles} in \cite{KM}, and
\textit{psyquandles} in \cite{NOS}. In particular, the mc-biquandle operations
are not defined geometrically but purely combinatorially. We do not have a good
geometric interpretation for these structures, but the same is true of standard
biquandles, psyquandles and many other coloring structures.
\end{remark}

\begin{remark}
Let $X$ be an mc-biquandle. If $\utr^m=\utr^s$ and $\otr^m=\otr^s$, 
then $X$ is a biquandle in the usual sense. More generally, an mc-biquandle 
is a pair consisting of a biquandle structure $\utr^s, \otr^s$
on a set $X$ and a birack structure $\utr^m,\otr^m$ on the same set which
are compatible in the way described by axiom (iii).
\end{remark}

\begin{example}\label{ex:te}
Let $X$ be a biquandle with two binary operations $\utr, \otr$ 
and set the single-component operations $\utr^s, \otr^s$ to be those of 
$X$ and the multi-component operations $\utr^m, \otr^m$ to be trivial, i.e.
\[x\utr^sy=x\utr y,\quad x\otr^s y=x\otr y, \quad
x\utr^my=x,\quad \mathrm{and}\quad x\otr^m y=x.\]
Then one can easily check that the mc-biquandle axioms are satisfied; let us
call the resulting mc-biquandle a \textit{trivial extension} of $X$.
The invariants we define below are insensitive to crossing changes at 
multicomponent crossings if they use this type of mc-biquandle, a sort of 
opposite situation from link homology where single-component crossing changes 
are permitted but not multi-component ones. Colorings by trivial extension 
also cannot distinguish a link $L$ from the virtual link obtained by
virtualizing (in the sense of making virtual) all classical multicomponent 
crossings.
\end{example}

\begin{example}
We can specify an mc-biquandle structure on a finite set $X=\{1,2,\dots, n\}$
with operations tables for each of the four operations, e.g.
\[
\begin{array}{r|rrr}
\utr^s & 1 & 2 & 3 \\ \hline
1 & 1 & 1 & 1 \\
2 & 2 & 2 & 2 \\
3 & 3 & 3 & 3
\end{array}
\quad
\begin{array}{r|rrr}
\otr^s & 1 & 2 & 3 \\ \hline
1 & 1 & 3 & 1 \\
2 & 2 & 2 & 2 \\
3 & 3 & 1 & 3
\end{array}
\quad
\begin{array}{r|rrr}
\utr^m & 1 & 2 & 3 \\ \hline
1 & 3 & 3 & 3 \\
2 & 2 & 2 & 2 \\
3 & 1 & 1 & 1 \\
\end{array}
\quad
\begin{array}{r|rrr}
\otr^m & 1 & 2 & 3 \\ \hline
1 & 1 & 1 & 1 \\
2 & 2 & 2 & 2 \\
3 & 3 & 3 & 3.
\end{array}
\]
We can compactify this notation by writing the four
operation tables as a block matrix of the format
$[\utr^s | \otr^s | \utr^m| \otr^m]$, e.g.,
\[
\left[
\begin{array}{rrr|rrr|rrr|rrr}
 1 & 1 & 1  & 1 & 3 & 1 &  3 & 3 & 3 &  1 & 1 & 1 \\
 2 & 2 & 2  & 2 & 2 & 2 &  2 & 2 & 2 &  2 & 2 & 2 \\
 3 & 3 & 3  & 3 & 1 & 3 &  1 & 1 & 1 &  3 & 3 & 3
\end{array}\right].
\]
\end{example}

\begin{definition}
Let $D$ be a diagram of an oriented link $L$. The \textit{fundamental mc-biquandle}
of $L$, denoted $\mathcal{MCB}(D)$, is defined in the following way:
\begin{itemize}
\item Each semiarc in $D$ has a corresponding generator $x_j$,
\item An \textit{mc-biquandle word} is a generator $x_j$ or an expression
of one of the forms
\[\{w_j\utr^lw_k,\ 
w_j\otr^lw_k,\ 
(\alpha_{w_j}^{l})^{-1}(w_k),\
(\beta_{w_j}^{l})^{-1}(w_k),\
(S^l)^{-1}_i(w_j,w_k)\}\]
where $l\in\{s,m\}$, $w_j,w_k$ are mc-biquandle words and 
$i\in\{1,2\}$ (and $S_1(x,y)=x, S_2(x,y)=y$),
\item The elements of $\mathcal{MCB}(D)$ are equivalence classes of 
mc-biquandle words under the congruence
generated by the mc-biquandle axioms and the crossing relations. 
\end{itemize}
\end{definition}


A presentation for an mc-biquandle consists of the generating set $S$ and the set of relations $R$, that is, it is the quotient of the free mc-biquandle on the set $S$ by equivalence relation generated by $R$.
It follows that the mc-biquandle $\mathcal{MCB}(D)$ has a presentation 
$$\langle ~ x_{1}, \cdots, x_{n} ~|~ 
r_{1}^{s}(c_1), r_{2}^{s}(c_1), \cdots, r_{1}^{s}(c_{n_{s}}), r_{2}^{s}(c_{n_{s}}), 
r_{1}^{m}(d_1), r_{2}^{m}(d_1), \cdots, r_{1}^{m}(d_{n_{m}}), r_{2}^{m}(d_{n_{m}}) ~\rangle$$
where $x_{1}, \cdots, x_{n}$ are the generators of the semiarcs of $D$, 
$c_1, \cdots, c_{n_{s}}$ are single-component crossings of $D$, 
$d_1, \cdots, d_{n_{m}}$ are multi-component crossings of $D$, and 
the relations $r_{1}^{s}(c_k), r_{2}^{s}(c_k), r_{1}^{m}(d_{k'}), r_{2}^{m}(d_{k'})$ are the crossing relations defined as 
\[r_{1}^{s}(c_k): x_{k_{3}}=x_{k_{1}}\utr^{s}x_{k_{2}}, \
r_{2}^{s}(c_k): x_{k_{4}}=x_{k_{2}}\otr^{s}x_{k_{1}}, \
r_{1}^{m}(d_{k'}): x_{k'_{3}}=x_{k'_{1}}\utr^{m}x_{k'_{2}}\ \mathrm{and}\ 
r_{2}^{m}(d_{k'}): x_{k'_{4}}=x_{k'_{2}}\otr^{m}x_{k'_{1}}\]
where the crossings are of the form
\[\includegraphics{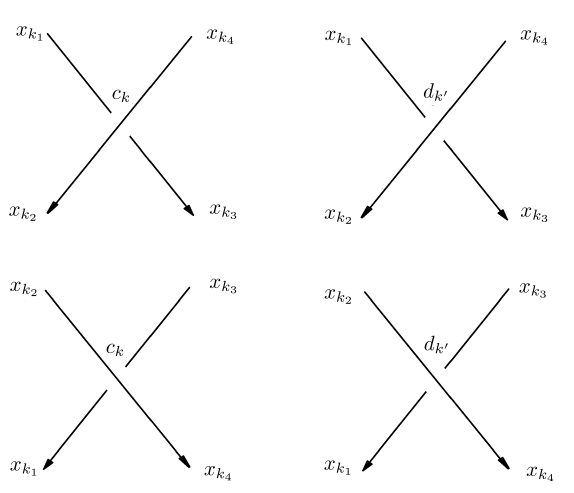}.\]
 
\begin{theorem}
Let $D$ and $D'$ be two diagrams which represent the same link $L$. 
Then the mc-biquandle $\mathcal{MCB}(D)$ is isomorphic to the mc-biquandle $\mathcal{MCB}(D')$.
\end{theorem}

\begin{proof}
We may assume that two diagrams $D$ and $D'$ are obtained from each other by 
a sequence of Reidemeister moves I, II, III. 
Then one sees that the resulting presentations $\mathcal{MCB}(D)$ and $\mathcal{MCB}(D')$ differ only by the condition (i), (ii), or (iii) of Definition 
\ref{Defmcbiqdle} 
corresponding to Reidemeister moves I, II, or III.
Therefore, they are isomorpic by Tietze transformations.
\end{proof}

\begin{definition}
Let $L$ be an oriented link and $D$ be its diagram. 
The \textit{fundamental mc-biquandle} $\mathcal{MCB}(L)$ of $L$ can be defined as $\mathcal{MCB}(D)$.
\end{definition}

%


\begin{example}
Let $R$ be a commutative ring with identity. A choice of units $t^s,t^m,r^s,r^m\in R$
defines an \textit{Alexander mc-biquandle} structure on $R$ by the operations
\[x\utr^l y = t^lx+(r^l-t^l)y\ \mathrm{and}\ x\otr^l y=r^l y\]
provided the equations
\begin{eqnarray*}
(r^j-r^k)(r^l-t^l) & = & 0 \\
(t^k-t^l)(r^j-t^j) & = & 0 \\
(r^j-t^j)(r^l-t^l) & = & (r^l-t^j)(r^k-t^k)
\end{eqnarray*}
are satisfied for all $(j,k,l)\in \{(s,s,s),(s,m,m), (m,s,m),(m,m,s),(m,m,m)\}$,
as the reader may verify from the mc-biquandle axioms.

A similar construction extending Alexander quandles rather than biquandles
can be found in a series of papers by L. Trialdi starting with \cite{T}.
\end{example}

\begin{definition} Let $(X,\utr^s,\otr^s,\utr^m,\otr^m)$ and 
$(Y,\utr^s,\otr^s,\utr^m,\otr^m)$ be mc-biquandles. A map $f:X\to Y$ is 
a \textit{homomorphism} if for all $x,y\in X$, we have
\[f(x\utr^s y)=f(x)\utr^s f(y),\
f(x\otr^s y)=f(x)\otr^s f(y),\
f(x\utr^m y)=f(x)\utr^m f(y),\ \mathrm{and}\
f(x\otr^m y)=f(x)\otr^m f(y).
\]
A self-homomorphism of an mc-biquandle is an \textit{endomorphism}.
\end{definition}

\begin{definition}
Let $L$ be a link and $X$ a finite mc-biquandle. The set of homomorphisms
\[\mathrm{Hom}(\mathcal{MCB}(L),X)=\{f:\mathcal{MCB}(L)\to X\}\]
is called the \textit{mc-biquandle homset} of $L$ with respect to $X$.
We denote its cardinality by 
\[\Phi^{\mathbb{Z}}_X(L)=|\mathrm{Hom}(\mathcal{MCB}(L),X)|.\]
\end{definition}

The mc-biquandle operations are the conditions required so that given
any mc-biquandle coloring of a classical or virtual link before a Reidemeister
move, there is a unique mc-biquandle coloring of the diagram after the move
which agrees with the initial coloring outside the neighborhood of the move.
Indeed, since an mc-biquandle coloring of a diagram determines images for
a generating set for the fundamental mc-biquandle of $L$, it follows that
such colorings uniquely determine homomorphisms $f:\mathcal{MCB}(L)\to X$.
Hence, the abstract homset $\mathrm{Hom}(\mathcal{MCB}(L),X)$ is 
an invariant of oriented links. A diagram $D$ of $L$ gives us a concrete
way to represent elements of the homset as \textit{colorings} of $D$, i.e.
assignments of elements of $X$ to the semiarcs in $D$ satisfying the coloring
rules at the crossings, analogous to the way choosing a basis provides concrete
matrix representations of linear transformations.

\begin{theorem}
The number $\Phi^{\mathbb{Z}}_X(L)$ of colorings of any diagram of $L$
by $X$, is an integer-valued invariant of oriented links known as the
\textit{mc-biquandle counting invariant}.
\end{theorem}

\begin{example}
Let $X=\mathbb{Z}_3=\{1,2,3\}$ be the integers modulo 3 with the class of 
zero represented by 3. Then $X$ is an mc-biquandle under the operations
\[\begin{array}{rclrcl}
x\utr^s y & = & 2x+2y & x\utr^m y & = & 2x \\
x\otr^s y & = & x & x\otr^m y & = & 2x. \\
\end{array}
\]
Then the virtual link $L$ with diagram
\[\includegraphics{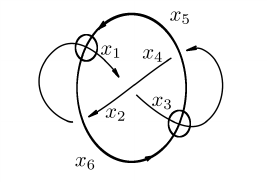}\]
has coloring matrix
\[
\left[\begin{array}{rrrrrr}
2 & 2 & 2 & 0 & 0 & 0  \\
0 & 1 & 0 & 2 & 0 & 0  \\
2 & 2 & 0 & 0 & 0 & 0  \\
0 & 0 & 0 & 0 & 2 & 2  \\
0 & 0 & 0 & 0 & 2 & 2  \\
0 & 0 & 2 & 2 & 0 & 0  \\
\end{array}\right]
\]
which row-reduces over $\mathbb{Z}_3$ to
\[
\left[\begin{array}{rrrrrr}
1 & 1 & 0 & 0 & 0 & 0  \\
0 & 1 & 0 & 2 & 0 & 0  \\
0 & 0 & 1 & 0 & 0 & 0  \\
0 & 0 & 0 & 1 & 0 & 0  \\
0 & 0 & 0 & 0 & 1 & 1  \\
0 & 0 & 0 & 0 & 0 & 0  \\
\end{array}\right]
\]
and hence the homset is isomorphic as a $\mathbb{Z}_3$-module
to $\mathbb{Z}_3$, and we have $\Phi_X^{\mathbb{Z}}(L)=|\mathbb{Z}_3|=3.$
\end{example}

\begin{example}
Our \textit{python} computations show that the links $L4a1$ and $L5a1$ 
have the same counting invariant values for all biquandles with three elements; 
however, the links are distinguished by the three-element mc-biquandle $X$ 
given by the operation tables
\[
\begin{array}{r|rrr}
\utr^s & 1 & 2 & 3 \\ \hline
1 & 2 & 2 & 2 \\
2 & 3 & 3 & 3 \\
3 & 1 & 1 & 1 \\
\end{array}
\quad
\begin{array}{r|rrr}
\otr^s & 1 & 2 & 3 \\ \hline
1 & 2 & 2 & 2 \\
2 & 3 & 3 & 3 \\
3 & 1 & 1 & 1 \\
\end{array}
\quad
\begin{array}{r|rrr}
\utr^m & 1 & 2 & 3 \\ \hline
1 & 1 & 1 & 1 \\
2 & 2 & 2 & 2 \\
3 & 3 & 3 & 3 \\
\end{array}
\quad
\begin{array}{r|rrr}
\otr^m & 1 & 2 & 3 \\ \hline
1 & 2 & 2 & 2 \\
2 & 3 & 3 & 3 \\
3 & 1 & 1 & 1 \\
\end{array}
\]
with $\Phi_X^{\mathbb{Z}}(L4a1)=0$ and $\Phi_X^{\mathbb{Z}}(L5a1)=9$.
\end{example}

\section{\large\textbf{MC-Biquandle Coloring Quivers}}\label{MCBQ}

We now categorify the mc-biquandle counting invariant using a construction
introduced in \cite{CN}.

\begin{definition}
Let $X$ be an mc-biquandle and $L$ an oriented classical or virtual knot or
link diagram. For any
endomorphism $\phi:X\to X$, if $f\in\mathrm{Hom}(\mathcal{MCB}(L),X)$ then
the composition $f\phi$ defines another element of 
$\mathrm{Hom}(\mathcal{MCB}(L),X)$. The \textit{mc-biquandle coloring quiver}
associated to a subset $S\subset\mathrm{End}(X)$ of the set of endomorphisms of
$X$ is the directed graph $\mathcal{Q}_X^S(L)$ with a vertex for each element 
of the homset
$\mathrm{Hom}(\mathcal{MCB}(L),X)$ and an edge from $f$ to $f\phi$ for each 
$\phi\in S$.
\end{definition}

The key insight is that if we have an mc-biquandle coloring of a
diagram and we apply an endomorphism $f:X\to X$ to each of the colors,
the result is another valid coloring:
\[\includegraphics{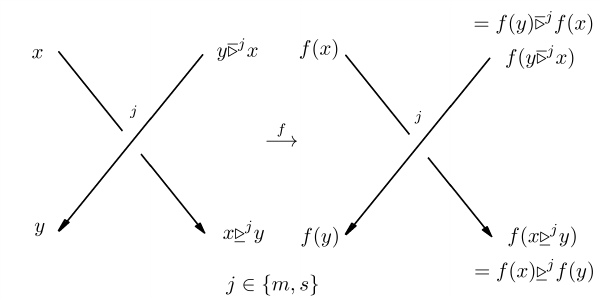}\]
Hence it follows that each endomorphism induces a directed edge from each 
vertex to another (possibly the same) vertex.

\begin{theorem}\label{th2}
Let $X$ be a finite mc-biquandle and $L$ an oriented classical or virtual 
knot or link diagram. Then the directed graph isomorphism type of 
$\mathcal{Q}_X^S(L)$ is an invariant of $L$.
\end{theorem}

\begin{proof}
The directed graph $\mathcal{Q}_X^S(L)$ is determined by 
$\mathrm{Hom}(\mathcal{MCB}(L),X)$ which is
invariant under Reidemeister moves.
\end{proof}

\begin{example}\label{ex:vwh}
Let $X$ be the mc-biquandle 
structure on the set $\{1,2,3\}$ given by the operation tables
\[
\begin{array}{r|rrr}
\utr^s & 1 & 2 & 3 \\ \hline
1 & 1 & 1 & 1 \\
2 & 2 & 2 & 2 \\
3 & 3 & 3 & 3
\end{array}
\quad
\begin{array}{r|rrr}
\otr^s & 1 & 2 & 3 \\ \hline
1 & 1 & 3 & 1 \\
2 & 2 & 2 & 2 \\
3 & 3 & 1 & 3
\end{array}
\quad
\begin{array}{r|rrr}
\utr^m & 1 & 2 & 3 \\ \hline
1 & 3 & 3 & 3 \\
2 & 2 & 2 & 2 \\
3 & 1 & 1 & 1 \\
\end{array}
\quad
\begin{array}{r|rrr}
\otr^m & 1 & 2 & 3 \\ \hline
1 & 1 & 1 & 1 \\
2 & 2 & 2 & 2 \\
3 & 3 & 3 & 3.
\end{array}
\]
Then we compute that endomorphisms of $X$ include the identity function 
$\phi_1$, the constant function $\phi_2$ sending everything to 2, and the 
bijection $\phi_3$ fixing 2 and switching 1 with 3; let us
denote these endomorphisms by listing the images of the elements in order,
i.e., $\phi_1=[1,2,3]$, $\phi_2=[2,2,2]$ and $\phi_3=[3,2,1]$. 
Then the virtual link 
below has three $X$-colorings; setting $S$ to be the full set of 
endomorphisms, we have full mc-biquandle coloring quiver as shown.
\[\includegraphics{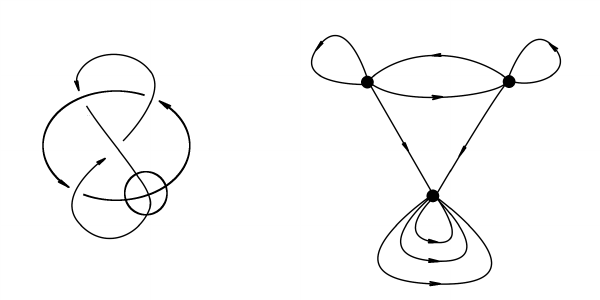}\]
\end{example}

\begin{example}
Let $X$ be the mc-biquandle structure on the set $\{1,2,3,4,5\}$ defined by
the operation tables
\[
\begin{array}{r|rrrrr}
\utr^s & 1 & 2 & 3 & 4 & 5 \\ \hline
1 & 1 & 3 & 2 & 1 & 1 \\
2 & 3 & 2 & 1 & 2 & 2 \\
3 & 2 & 1 & 3 & 3 & 3 \\
4 & 4 & 4 & 4 & 4 & 4 \\
5 & 5 & 5 & 5 & 5 & 5
\end{array}
\quad
\begin{array}{r|rrrrr}
\otr^s & 1 & 2 & 3 & 4 & 5 \\ \hline
1 & 1 & 1 & 1 & 1 & 1 \\
2 & 2 & 2 & 2 & 2 & 2 \\
3 & 3 & 3 & 3 & 3 & 3 \\
4 & 4 & 4 & 4 & 4 & 4 \\
5 & 5 & 5 & 5 & 5 & 5
\end{array}
\quad
\begin{array}{r|rrrrr}
\utr^m & 1 & 2 & 3 & 4 & 5\\ \hline
1 & 1 & 3 & 2 & 1 & 1 \\
2 & 3 & 2 & 1 & 2 & 2 \\
3 & 2 & 1 & 3 & 3 & 3 \\
4 & 5 & 5 & 5 & 5 & 5 \\
5 & 4 & 4 & 4 & 4 & 4 
\end{array}
\quad
\begin{array}{r|rrrrr}
\otr^m & 1 & 2 & 3 & 4 & 5\\ \hline
1 & 1 & 1 & 1 & 1 & 1 \\
2 & 2 & 2 & 2 & 2 & 2 \\
3 & 3 & 3 & 3 & 3 & 3 \\
4 & 4 & 4 & 4 & 4 & 4 \\
5 & 5 & 5 & 5 & 5 & 5
\end{array}
\]
and consider the 2-component virtual links shown:
\[\includegraphics{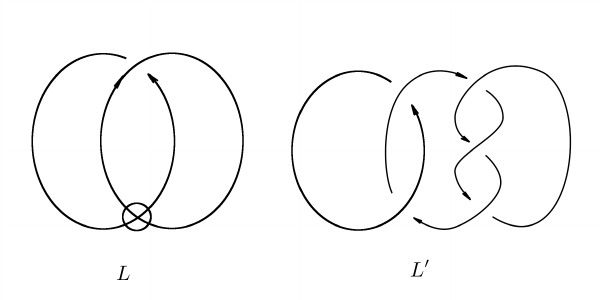}\]
$X$ has 21 endomorphisms, including for example $[3,3,3,4,5]$. 
Then we compute that both links $L$ and $L'$ have $9$ $X$-colorings, but are
distinguished by their quivers $\mathcal{Q}_X^S(L)$ and $\mathcal{Q}_X^S(L')$ :
\[\includegraphics{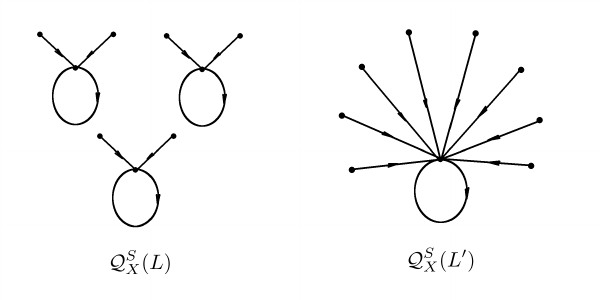}\]
\end{example}

Because comparing quivers directly to determine isomorphism is 
generally computationally expensive, we would like to
extract invariants of the quiver. Since the quiver is unchanged by
Reidemeister moves, any such quiver invariants are automatically 
(classical and virtual) link invariants.
Since the quiver is a categorification of the counting invariant, 
these invariants are decategorifications of the quiver.

We can define a polynomial invariant obtained from a mc-biquandle coloring quiver, as in \cite{CN}.

\begin{definition}
Let $X$ be an mc-biquandle and $S\subset \mathrm{End}(X)$.
We define the \textit{in-degree polynomial} of an oriented link $L$ with respect
to $X$ and $S$ to
be the polynomial
\[\Phi_X^{S,\mathrm{deg}^+}(L)
=\sum_{f\in\mathrm{hom}(\mathcal{MCB}(L),X)} u^{\mathrm{deg}_+(f)}\] 
where $\mathrm{deg}_+(f)$ is the in-degree of the vertex associated to $f$ 
in the mc-biquandle coloring quiver $\mathcal{Q}_X^S(L)$.
\end{definition}

We note that evaluation of the in-degree polynomial at $u=1$ yields
the cardinality of $S$ times the counting invariant.

\begin{theorem}
Let $D$ and $D'$ be diagrams representing the same oriented classical or 
virtual knot or link and let $X$ be a finite mc-biquandle. Then 
\[\Phi_X^{S,\mathrm{deg}^+}(D)=\Phi_X^{S,\mathrm{deg}^+}(D')\]
\end{theorem}

\begin{proof}
This follows immediately from Theorem \ref{th2}.
\end{proof}

\begin{example}
The virtual link in Example \ref{ex:vwh} above has in-degree polynomial
$2u^2+u^5$ with respect to the mc-biquandle $X$ and full set $S$ of 
endomorphisms in Example \ref{ex:vwh}.
\end{example}

\begin{example}
Let $X$ be the mc-biquandle structure on the set $\{1,2,3,4\}$ given by the
operation tables
\[
\begin{array}{r|rrrr}
\utr^s & 1 & 2 & 3 & 4 \\ \hline
1 & 2 & 2 & 1 & 1 \\
2 & 1 & 1 & 2 & 2 \\
3 & 4 & 4 & 3 & 3 \\
4 & 3 & 3 & 4 & 4 
\end{array}\quad
\begin{array}{r|rrrr}
\otr^s & 1 & 2 & 3 & 4 \\ \hline
1 & 2 & 2 & 2 & 2 \\
2 & 1 & 1 & 1 & 1 \\
3 & 3 & 3 & 3 & 3 \\
4 & 4 & 4 & 4 & 4
\end{array}\quad
\begin{array}{r|rrrr}
\utr^m & 1 & 2 & 3 & 4 \\ \hline
1 & 1 & 1 & 2 & 2 \\
2 & 2 & 2 & 1 & 1 \\
3 & 4 & 4 & 3 & 3 \\
4 & 3 & 3 & 4 & 4
\end{array}\quad
\begin{array}{r|rrrr}
\otr^m & 1 & 2 & 3 & 4 \\ \hline
1 & 1 & 1 & 1 & 1 \\
2 & 2 & 2 & 2 & 2 \\
3 & 3 & 3 & 3 & 3 \\
4 & 4 & 4 & 4 & 4
\end{array};
\]
we compute that the endomorphisms of $X$ are
\[\{[1, 2, 3, 4], [1, 2, 4, 3], [2, 1, 3, 4], [2, 1, 4, 3], 
[3, 3, 3, 3], [3, 3, 4, 4], [4, 4, 3, 3], [4, 4, 4, 4]\}\]
where we express an endomorphism $f:X\to X$ in the form $[f(1),\dots, f(n)]$.

Using our custom \texttt{python} code, we computed the in-degree polynomial
values for the prime classical links with up to 7 crossings as listed in
the knot atlas \cite{KA}:
\[
\begin{array}{r|l}
\Phi_X^{S,\mathrm{deg}_+}(L) & L \\\hline
2u^{20}+6u^4 & L2a1, L6a2, L6a3, L7a5, L7a6 \\
12u^4 + 2u^{12} + 2u^{28} & L4a1, L5a1, L6a1, L7a1, L7a2, 
L7a3, L7a4, L7n1, L7n2 \\
14u^4+2u^{36} & L6a5, L6n1, L7a7 \\
56u^4+6u^{20}+2u^{84} & L6a4 \\
\end{array}
\]
In particular, we observe that this example shows that this enhancement is 
proper, i.e., can distinguish links with the same counting invariant.
\end{example}

\section{\large\textbf{Questions}}\label{Q}

We conclude with a few questions and direction for future research.

In a future paper we will extend the notion of biquandle brackets to the case 
of mc-biquandles analogously to Kaestner brackets \cite{KN}. Extending
biquandle homology to the case of mc-biquandles is also a direction of 
interest. What other enhancements of the mc-biquandle counting invariant 
can be defined? What other decategorifications of the quiver can be found?

\bibliography{sc-sn3}{}
\bibliographystyle{abbrv}

\noindent
\textsc{Nonlinear Dynamics and Mathematical Application Center and \\
Department of Mathematics, \\
Kyungpook National University \\
Daegu, 41566, Republic of Korea} 

\bigskip

\noindent
\textsc{Department of Mathematical Sciences \\
Claremont McKenna College \\
850 Columbia Ave. \\
Claremont, CA 91711}

\end{document}